\documentclass[11pt]{amsart}

\pdfoutput=1

\usepackage[utf8]{inputenc}
\usepackage[T1]{fontenc}
\usepackage{lmodern}

\usepackage{ulem,cancel}
\usepackage{amsthm, amssymb, amsmath, amsfonts, mathrsfs}

\usepackage{microtype}

\usepackage[pagebackref,colorlinks=true,pdfpagemode=none,urlcolor=blue,
linkcolor=blue,citecolor=blue]{hyperref}

\usepackage{amsmath,amsfonts,amssymb,amsthm}
\usepackage{amsthm, amssymb, amsmath, amsfonts, mathrsfs}

\usepackage{mathrsfs}
\usepackage{MnSymbol}
\usepackage{scalerel} 

\usepackage{color}
\usepackage{accents}

\usepackage{bbm}

\definecolor{labelkey}{gray}{.8}
\definecolor{refkey}{gray}{.8}

\definecolor{darkred}{rgb}{0.9,0.1,0.1}
\definecolor{darkgreen}{rgb}{0,0.5,0}

\newtheorem{theorem}{Theorem}[section]
\newtheorem{lemma}[theorem]{Lemma}

\newtheorem{proposition}[theorem]{Proposition}

\theoremstyle{remark}
\newtheorem{remark}[theorem]{Remark}

\renewenvironment{proof}[1][Proof]{ {\itshape \noindent {#1.}} }{$\Box$
\medskip}

\numberwithin{equation}{section}
\newcommand{\R}{\mathbb{R}}

\newcommand{\Pb}{\mathbb{P}}
\newcommand{\PP}{\mathbf{P}}
\newcommand{\E}{\mathbb{E}}

\newcommand{\F}{\mathcal{F}}

\newcommand{\W}{\mathscr{W}}

\newcommand{\D}{\mathcal{D}}

\newcommand{\Var}{\mathrm{Var}}

\newcommand{\la}{\langle}
\newcommand{\ra}{\rangle}

\newcommand{\X}{\mathbf{X}}

\newcommand{\Y}{\mathbf{Y}}

\newcommand{\EE}{\mathbf{E}}
\newcommand{\1}{\mathbbm{1}}

\newcommand{\cH}{\mathcal{H}}

\newcommand{\cov}{\mathrm{Cov}}

\newcommand{\cZ}{\mathcal{Z}}

\newcommand{\cX}{\mathcal{X}}
\newcommand{\ssp}{\mathsf{p}}
\newcommand{\cD}{\mathscr{D}}

\begin{document}

\title[Variance of stationary KPZ]{Another look at the Bal\'azs-Quastel-Sepp\"al\"ainen theorem}
\author{Yu Gu and Tomasz Komorowski}

\address[Yu Gu]{Department of Mathematics, University of Maryland, College Park, MD 20742, USA}

\address[Tomasz Komorowski]{Institute of Mathematics, Polish Academy
  of Sciences, ul. \'{S}niadeckich 8, 00-656, Warsaw, Poland}

\maketitle

\begin{abstract}

We study the KPZ equation with a $1+1-$dimensional spacetime white noise, started
at equilibrium, and give a different proof of the main result of
\cite{bqs}, i.e., the variance of the solution at time $t$ is of order $t^{2/3}$. Instead of using a discrete approximation through the exclusion process and the second class particle, we utilize the connection to  directed polymers in random environment. Along the way, we show  the annealed density of the stationary continuum directed polymer equals to the two-point covariance function of the stationary stochastic Burgers equation, confirming the physics prediction in \cite{MT}.
\bigskip




\end{abstract}
\maketitle

\section{Main result}

Consider the stochastic heat equation (SHE) started from the exponential of a  drifted two-sided Brownian motion:
\begin{equation}\label{e.she}
\begin{split}
&\partial_t Z_\theta(t,x)=\frac12\Delta Z_\theta(t,x)+\xi(t,x)
Z_\theta(t,x) , \quad (t,x)\in(0,\infty)\times\R,\\
&Z_\theta(0,x)=e^{W_\theta(x)}.
\end{split}
\end{equation}
Here $W_\theta(x)=W(x)+\theta x$, and $W$ is  a two-sided Brownian
motion with $W(0)=0$ and $\theta\in\R$ is an arbitrary constant. 
The noise $\xi$ is a space-time white noise, i.e., it is a generalized
Gaussian random field with the covariance function $\EE\,
\xi(t,x)\xi(s,y)=\delta(t-s)\delta(x-y)$. Both the noise $\xi$ and the
Brownian motion $W$ are defined over some probability space
$(\Omega,{\mathcal F},\PP)$, with $\EE$ denoting the expectation. 

Define $h_\theta(t,x)=\log Z_\theta(t,x)$. The following is the main result:
\begin{theorem}\label{t.bqs}
There exists a constant $C>1$ such that 
\begin{equation}\label{e.maingoal}
C^{-1} t^{\frac{2}{3}} \leq  \Var \, h_0(t,0) \leq C t^{\frac{2}{3}}, \quad\quad \mbox{ for } t\geq1.
\end{equation}
\end{theorem}

\subsection{Context}

The study of the KPZ equation with a $1+1$ spacetime white noise has
witnessed tremendous progress during the past decade. One of the main
achievements is to show that, under the 1:2:3 scaling and after a
centering, the solution converges in law to the KPZ fixed point, that is, the Markov process, which is expected to be the limit of all models in the 1+1 KPZ universality class, see \cite{QS20,Vir20,MQR20} for the related results and \cite{corwin2012kardar,Qua12,qs} for reviews and surveys in this area. Despite important progresses, many problems remain, in particular, how to extend the existing results to non-integrable models is of great interest. 

 Many studies on the KPZ equation rely on connections to discrete
 models, in particular the asymmetric simple exclusion process (ASEP), see e.g. \cite{BG97,bqs,SS10a,acq}. In this paper, we revisit an ``old'' problem:  it was shown in \cite{bqs} that the solution to the KPZ equation, started at equilibrium, has a $t^{1/3}$ size of fluctuations in large time (see \cite{corwinh} for the results on general initial data). The proof in \cite{bqs} relied on the study of the second class particle in ASEP. We provide a different proof here, through a connection to the directed polymer instead. As a crucial ingredient, we will derive a variance identity and show that the variance of the height function equals to the first moment of the endpoint of the continuum directed polymer at stationarity. This, combined with the result in \cite{bqs}, shows the two-point covariance function of the stochastic Burgers equation at stationarity actually coincides with the annealed density function of the endpoint of the directed polymer, see Remark~\ref{r.polymer}. This was conjectured in the physics literature \cite{MT}. Our proof is based on an integration by parts in the Gaussian space induced by the two-sided Brownian motion. Similar strategies have actually been adopted to study the KPZ fixed point started at equilibrium \cite{pimentel}. See the more recent development in \cite{pimentel1}.

Using the aforementioned variance identity, the study of the fluctuations of the height function reduces to that of the endpoint of the directed polymer. A few directed polymer models are shown to be in the $1+1$ KPZ
universality class, see
e.g. \cite{timo,timo1,timo2,acq,bcr,bc,bcf,BCFV15,haoshen,bc1,Vir20} for relevant 
results of proving the  scaling exponents, deriving the Tracy-Widom
type fluctuations etc. Our proof of the upper bound  is inspired by
the approach used to study the O'Connell-Yor polymer in \cite{timo1},
which was further explored in \cite{soso1} (see the recent study
on the interacting diffusions \cite{soso2} using a similar strategy). The key is to make use of
the convexity of the function $h_\theta(t,0)$ in the
$\theta-$variable, and the statistical invariance of the driving noise
under shear transformations, which leads to the quadratic form of the free energy $\EE h_\theta(t,0)$, as a function of $\theta$, see Proposition~\ref{l.meanhtheta} below.  For the lower bound, we apply a similar approach as \cite{bqs}, which was inspired by   \cite{timo3}, where a similar result for ASEP was derived. The main coupling argument used in \cite[Lemma 4.1]{timo3} was replaced by Lemma~\ref{l.keyl} below. An advantage of directly studying the SHE or KPZ equation is to  apply the comparison principle, namely, if we start the equation with ordered initial conditions and drive the equation by the same noise, then the solutions are also ordered. 

The main point here is to provide a somewhat different and simpler proof of the seminal results in \cite{bqs}. Although more precise information was obtained later, see e.g. \cite[Theorem 1.2]{BCFV15} for the convergence in distribution of the rescaled random fluctuations, we are hoping that a different perspective  could be of independent interest.

To see the connection to the directed polymer more clearly, we write the solution to \eqref{e.she} through a formal Feynman-Kac formula as
\[
Z_{\theta}(t,x)=\E_B[\exp(\int_0^t \xi(t-s,B_s)ds)\exp(W_\theta(B_t))\,|\,B_0=x].
\]
Here $B$ is a standard Brownian motion that is independent of $(\xi,W)$ and $\E_B$ is the expectation on $B$ only. The above expression can be viewed as the partition function of  a directed polymer in the random environment $\xi$, with the boundary condition $W_\theta$. In other words, the polymer measure is the Wiener measure reweighted by the exponential factor $\exp(\int_0^t \xi(t-s,B_s)ds)\exp(W_\theta(B_t))$. Note that it is only a formal expression here since $\xi$ is a space-time white noise -- we will give a rigorous meaning of it in Section~\ref{s.direct} below.

A common feature of our proof and that of \cite{bqs} is to employ the variance identity which relates the height function and the displacement of the directed polymer. Similar identities appeared in other solvable models, see \cite[Theorem 3.6]{timo1} and \cite[Lemma 4.6]{timo4}. The difference is that, we will derive the identity directly on the level of the SHE, while \cite{bqs} used a discrete counterpart. Our proof through an integration by parts relies heavily on the Gaussian nature of the invariant measure. For the SHE with a colored noise, the existence/uniqueness of the invariant measure was shown in \cite{bakhtin,dunlap}, but we do not know whether there is a similar variance identity.

The rest of the paper is organized as follows. In Section~\ref{s.direct}, we use a Gaussian integration by parts to show that the variance of the height function is related to the displacement of the polymer endpoint.  In Sections~\ref{s.upbd} and \ref{s.lowbd}, we prove the upper and the lower bounds in \eqref{e.maingoal} separately.

Throughout the paper, we use $\int$ as a shorthand notation for $\int_{\R}$ and $\|\cdot\|_p$ to denote the norm of $L^p(\Omega,{\mathcal F},\PP)$ for any $p\geq1$.

\subsection*{Acknowledgements}
Y.G. was partially supported by the NSF through DMS-2203014. 
 T.K. acknowledges the support of NCN grant 2020/37/B/ST1/00426.   We thank the two anonymous referees for multiple suggestions which helped to improve the presentation

\section{Continuum directed polymer}
\label{s.direct}

In this section, through a Gaussian integration by parts, we rewrite $\Var\, h_0(t,0)$ as the first absolute moment of a directed polymer in random environment. To state the main result, we first introduce some notations.

Let $\cZ_t(x,y)$ be the Green's function of \eqref{e.she}, i.e., for any $y\in\R$,  
\[
\begin{aligned}
&\partial_t \cZ_t(x,y)=\frac12\Delta_x \cZ_t(x,y)+\xi(t,x)\cZ_t(x,y), \quad\quad (t,x)\in (0,\infty)\times \R,\\
&\cZ_0(x,y)=\delta(x-y).
\end{aligned}
\]
Define the quenched density of the directed polymer starting from $(t,x)$ and running backwards in time as 
\begin{equation}\label{e.defptheta}
p_\theta^x(t,y)=\frac{\cZ_t(x,y)e^{W_\theta(y)}}{\int \cZ_t(x,y')e^{W_\theta(y')}dy'}.
\end{equation}
We denote the endpoint of the polymer path by $B_t$, and let
$\Pb_\theta^x$ be the annealed  probability on the   endpoint, 
i.e., 
\begin{equation}\label{e.defPbtheta}
\Pb_\theta^x(B_t\in A)= \EE \int_A p_\theta^x (t,y)dy=\EE  \frac{\int_A\cZ_t(x,y)e^{W_\theta(y)}dy}{\int \cZ_t(x,y')e^{W_\theta(y')}dy'}.
\end{equation}
The expectation under $\Pb_\theta^x$ will be denoted by
$\E_\theta^x$. We shall  mostly focus our attention  on $p_0^0$, so to simplify the notation we use $\ssp=p_0^0$.

 Before presenting the main result of the section, we recall an elementary fact about the KPZ equation and the directed polymer.  Using the Green's function and the definition of $\ssp$, we can write 
\begin{equation}\label{e.9191}
\begin{aligned}
h_\theta(t,0)&= \log Z_\theta(t,0)=\log \int \cZ_t(0,y)e^{W(y)+\theta y}dy\\
&=\log \int \frac{\cZ_t(0,y)e^{W(y)+\theta y}}{\int\cZ_t(0,y')e^{W(y')}dy'} dy +\log \int\cZ_t(0,y')e^{W(y')}dy'\\
&=\log \int \ssp(t,y)e^{\theta y}dy+\log \int\cZ_t(0,y')e^{W(y')}dy'.
\end{aligned}
\end{equation}
In the first expression on the r.h.s., we note that $y$ is the variable corresponding to the endpoint of the directed polymer, $\theta$ is the dual variable, and $\int \ssp(t,y)e^{\theta y}dy$ is a moment generating function indexed by $\theta\in\R$. Thus, $\partial_\theta^n h_\theta(t,0)\,|_{\theta=0}$ is the corresponding $n-$th cumulant of the density $\ssp(t,\cdot)$.

\bigskip
The main result of this section is the following variance identity:
\begin{proposition}\label{p.inte}
For any $t>0$, we have 
\begin{equation}
\label{050403-22}
\Var\,  h_0(t,0)=\EE \int  |y| \ssp(t,y)dy=\E_0^0|B_t|.
\end{equation}
\end{proposition}

\begin{remark}\label{r.polymer}
It was shown in \cite[Proposition 3.1]{bqs} that $\Var\, h_0(t,0)=\int |y| S(t,dy)$, with $S(t,dy)$ the symmetric probability measure which is the space-time correlation measure of the stochastic Burgers equation, see \cite[Proposition 1.4]{bqs}. Combining with the above result, we conclude that $S(t,dy)=\EE \ssp(t,y)dy$, which was conjectured and proved nonrigorously in the physics literature, see \cite[Eq. (16)]{MT}. As a matter of fact, applying a proof similar to that of Lemma~\ref{l.inte} below, one can directly show that for any test functions $f,g$, 
\[
\EE \int f'(x)h_0(t,x) \int g'(y)h_0(0,y)dy=\int f(x)g(y)\EE\, \ssp(t,x-y)dxdy,
\]
which implies that on a formal level we have
  $$\EE \Big[\partial_x
h_0(t,x)\partial_x h_0(0,y)\Big]=\EE\, \ssp(t,x-y).
$$
\end{remark}

To prove the above proposition, we start with a few lemmas. First, for
two random variables $X$ and $Y$ over $\Omega$, we let $\cov\,[X,Y]$ denote
their covariance.
Define also 
\begin{equation}
\cH(t,x)=h_0(t,x)-h_0(0,x)=h_0(t,x)- W(x).
\end{equation}
\begin{lemma}
For any $t,x\geq 0$ we have
\begin{equation}\label{e.var2}
\begin{aligned}
\Var \,h_0(t,0)
= \cov[ \cH(t,0)-\cH(t,x), W(x)]+\cov[\cH(t,x),\cH(t,0)].
\end{aligned}
\end{equation}
\end{lemma}

\begin{proof}
For any $x\geq 0$, we start from the elementary identity
\begin{equation}\label{e.var1}
\begin{aligned}
&\Var [h_0(t,x)-h_0(t,0)]=\Var[ \cH(t,x)-\cH(t,0)+ W(x)]\\
&=\Var[\cH(t,x)-\cH(t,0)]+ x+2 \cov[ \cH(t,x)-\cH(t,0), W(x)].
\end{aligned}
\end{equation}
Through the Green's function of SHE, we can write 
\begin{equation}\label{e.htxW}
\begin{aligned}
\cH(t,x)
&=\log \int_{\R}\cZ_{t}(x,y)e^{W(y)-W(x)}dy\\
&=\log \int_{\R} \cZ_{t}(x,x+y)e^{W(x+y)-W(x)}dy,
\end{aligned}
\end{equation}
which implies that, for each fixed $t>0$ and as a process indexed by $x$, $\{\cH(t,x)\}_{x\in\R}$ is stationary. In particular, we have 
\[
\Var[\cH(t,x)-\cH(t,0)]=2\Var\, h_0(t,0)-2\cov[\cH(t,x),\cH(t,0)].
\]
Thus, \eqref{e.var1} becomes 
\begin{equation}\label{e.var3}
\begin{aligned}
\Var \,h_0(t,0)&=\frac12\Var[h_0(t,x)-h_0(t,0)]-\frac12 x\\
&+ \cov[ \cH(t,0)-\cH(t,x), W(x)]\\
&+\cov[\cH(t,x),\cH(t,0)].
\end{aligned}
\end{equation}
By the invariance of $W$, i.e. the fact that
$\{h_0(t,x)-h_0(t,0)\}_{x\in\R}$ is a two-sided Brownian motion for
any $t>0$ (see \cite{funaki}), we conclude that the first line on the r.h.s. of \eqref{e.var3} is zero, which completes the proof.
\end{proof}

Note that the second term on the r.h.s. of \eqref{e.var2} is the
  covariance function of $\cH(t,\cdot)$. For any fixed $t>0$, the
  strong correlation has not kicked in yet so one naturally expect the
  random field to decorrelate on a large distance, i.e., 
\begin{equation}\label{e.decayco}
\lim_{|x|\to \infty}\cov[\cH(t,x),\cH(t,0)]=0.
\end{equation}
Indeed, this was proved in \cite[Proposition 5.2]{bqs}. We will
provide a self-contained proof of \eqref{e.decayco} through an
application of the Gaussian-Poincar\'e covariance inequality. Since
this holds  only for  finite time and does not involve any KPZ behavior, we leave it to the appendix.

%

Given \eqref{e.decayco}, to prove Proposition~\ref{p.inte} we only need to show  
\begin{lemma}\label{l.inte}
As $|x|\to\infty$, we have 
\[
\cov[ \cH(t,0)-\cH(t,x), W(x)]\to  \EE \int |y|\ssp(t,y)dy,\quad t>0.
\]
\end{lemma}
 
The proof of the above lemma is through an integration by parts in the Gaussian space. Note that we have two Gaussian processes here, the noise $\xi$ and the two-sided Brownian motion $W$. We will perform an integration by parts on $W$, for each realization of  $\xi$.  

We first introduce some notations. Let $\W$ be the spatial white noise associated with $W$, i.e., in the distributional sense we have $\W(x)=W'(x)$. 
For any $\varphi\in L^2(\R)$, we write $\W(\varphi):=\int \varphi(z)\W(z)dz$, which is the usual Wiener integral. In this way, for $x>0$, we write $W(x)=\W(\1_{[0,x]}(\cdot))$. Let $\D$ be the Malliavin derivative with respect to $\W$. For a random variable $X$ that is a smooth functional  of $\W$, $\D X$ is an $L^2(\R)-$valued random variable, which we write  as $\D X=(\D_r X)_{r\in\R}$, and one interprets $\D_r X$ as the derivative of $X$ with respect to $\W(r)$. For an introduction to   Malliavin calculus, we refer to \cite[Chapter 1]{nualart}.

The following lemma is the key to link the variance of the $h_0$ to the density of the continuum directed polymer. 
\begin{lemma}\label{l.htzwx}
For any $z\geq0$ and $t,x>0$, we have 
\[
\begin{aligned}
\cov[\cH(t,z),W(x)]=\EE  \int \ssp(t,y)\1_{\{z+y>0\}}\min(x,z+y)dy- \min(x,z).
\end{aligned}
\]
\end{lemma}

\begin{proof}
Recall that $\cH(t,z)=h_0(t,z)-W(z)$, so we only need to consider the
covariance of $h_0(t,z)$ and $W(x)$. For every realization of $\xi$,
by the integration by parts formula, see e.g. \cite[(1.42), p. 37]{nualart}, we have 
\begin{equation}
\begin{aligned}
\cov[h_0(t,z),W(x)]=\EE[W(x)h_0(t,z)]&=\EE[\W(\1_{[0,x]}(\cdot)) h_0(t,z)]\\
&=\EE\, \la \1_{[0,x]}(\cdot),\D h_0(t,z)\ra.
\end{aligned}
\end{equation}
Here $\la\cdot,\cdot\ra$ is the inner product in $L^2(\R)$. Using the expression
\[
h_0(t,z)=\log \int \cZ_{t}(z,y)e^{ W(y)} dy,
\]
we have for any $r\geq0$ that
\[
\D_r h_0(t,z)= \frac{\int_0^\infty \cZ_{t}(z,y)e^{ W(y)}\1_{[0,y]}(r)dy }{\int \cZ_{t}(z,y)e^{ W(y)}dy}.
\]
This, in turn, implies 
\[
\begin{aligned}
\la \1_{[0,x]}(\cdot),\D h_0(t,z)\ra&=\int_0^x \D_r h_0(t,z)dr\\
&= \frac{\int  \cZ_{t}(z,y)e^{ W(y)}\1_{\{y>0\}}\min(x,y)dy}{ \int \cZ_{t}(z,y)e^{ W(y)}dy}\\
&= \frac{\int \cZ_{t}(z,z+y)e^{W(z+y)-W(z)}\1_{\{z+y>0\}}\min(x,z+y)dy}{ \int  \cZ_{t}(z,z+y)e^{W(z+y)-W(z)}dy}.
\end{aligned}
\]
Here in the last ``='' we changed variable $y\mapsto y+z$.  Taking
expectation, using the stationarity and the definition of
$\ssp(t,\cdot)$ (see \eqref{e.defptheta}), 
we have 
\[
\begin{aligned}
\EE\la \1_{[0,x]}(\cdot),\D h_0(t,z)\ra&= \EE \int \ssp(t,y)\1_{\{z+y>0\}}\min(x,z+y)dy, \\
\end{aligned}
\]
which completes the proof.
\end{proof}
 
Using the previous lemma we can complete the proof of Lemma~\ref{l.inte} hence that of Proposition~\ref{p.inte}:

\bigskip

\begin{proof}
By Lemma~\ref{l.htzwx}, we have 
\[
\begin{aligned}
\cov[ \cH(t,0)-\cH(t,x), W(x)]=&\EE \int  \ssp(t,y)\1_{\{y>0\}}\min(x,y)dy\\
&- \EE \int  \ssp(t,y)\1_{\{x+y>0\}}\min(x,x+y)dy+x.
\end{aligned}
\]
By the fact that $\EE \ssp(t,\cdot)$ is an even probability density, we can rewrite it as
\begin{equation}
\begin{aligned}
&\cov[\cH(t,0)-\cH(t,x),W(x)]\\
&=2 \EE \int_{0}^x y\ssp(t,y)dy+2 x\EE \int_x^\infty\ssp(t,y)dy.
\end{aligned}
\end{equation}
Applying Lemma~\ref{l.bdheatk} below, we complete the proof.
\end{proof}

\begin{lemma}\label{l.bdheatk}
For any $t>0$, there exists $C_t>0$ so that 
\[
\EE \,\ssp(t,x) \leq C_t \exp(-x^2/C_t), \quad\quad x\in\R. 
\]
\end{lemma}

\begin{proof}
First, we   write, by the Cauchy-Schwarz inequality,
\[
\begin{aligned}
\EE \, \ssp(t,x)&=\EE \frac{\cZ_t(0,x)e^{W(x)}}{\int \cZ_t(0,x')e^{W(x')}dx'}\\
&\leq  \EE (\int \cZ_t(0,x')e^{W(x')}dx')^{-2} \EE \cZ_t(0,x)^2 e^{2W(x)}. 
\end{aligned}
\]
For the first expectation, by Jensen's inequality we derive 
\[
(\int \cZ_t(0,x')e^{W(x')}dx')^{-2} \leq (\int \cZ_t(0,x'')dx'')^{-2}\int \cZ_t(0,x')e^{-2W(x')}dx',
\]
which implies 
\[
\begin{aligned}
&\EE (\int \cZ_t(0,x')e^{W(x')}dx')^{-2} \\
&\leq \sqrt{\EE (\int \cZ_t(0,x'')dx'')^{-4}}\int \|\cZ_t(0,x')\|_4\|e^{-2W(x')}\|_4dx'.
\end{aligned}
\]
We also have 
\[
\EE \cZ_t(0,x)^2 e^{2W(x)}=e^{2|x|}\EE \cZ_t(0,x)^2.
\]
Then the proof is completed by invoking the following negative and positive moment estimates: for any $p\geq1$, 
\[
\begin{aligned}
\EE (\int \cZ_t(0,x)dx)^{-p} \leq C_{t,p}, \quad\quad \|\cZ_t(0,x)\|_p\leq C_{t,p}\exp(-x^2/C_{t,p}),
\end{aligned}
\]
see \cite[Corollary 4.8]{khoa} and \cite[Theorem 2.4, Example 2.10]{chenle} respectively.
\end{proof}

At the end of this section, we present the following result which will be used frequently. 
\begin{proposition}\label{l.meanhtheta}
(i) For any $t>0,x\in\R$, we have 
\[
\EE h_\theta(t,x)=\EE h_0(t,0)+\theta x+\frac12\theta^2 t.
\]
(ii) For any $t>0, \theta\in\R$, we have 
\[
\sqrt{\Var\,h_\theta(t,0)} \leq \sqrt{\Var\, h_0(t,0)}+ \sqrt{|\theta|t}.
\]
(iii) For any $\theta,x\in\R, t>0$ and $A\subset \R$, we have 
\[
\Pb_\theta^x(B_t\in A)=\Pb_0^0(x+B_t+\theta t\in A).
\]
(iv) For any  $t>0$ the function $\theta\mapsto h_\theta(t,0)$, $\theta\in\R$
is convex.

\end{proposition}

\begin{proof}
The result is rather standard, so we only sketch the argument.  

(i) Recall that $Z_\theta$ solves \eqref{e.she}, with $Z_\theta(0,x)=e^{W_\theta(x)}$ and $h_\theta=\log Z_\theta$. We claim 
\begin{equation}\label{e.331}
\{Z_\theta(t,x)\}_{t>0,x\in\R}\stackrel{\text{law}}{=}\{Z_0(t,x+\theta t)e^{\theta x+\frac12\theta^2 t}\}_{t>0,x\in\R},
\end{equation}
which comes from the fact that $Z_0(t,x+\theta t)e^{\theta x+\frac12\theta^2 t}$ solves \eqref{e.she} with $\{\xi(t,x)\}$ replaced by $\{\xi(t,x+\theta t)\}$ and the two random fields have the same distribution. With \eqref{e.331}, we have $\EE h_\theta(t,x)=\EE h_0(t,x+\theta t)+\theta x+\frac12\theta^2 t$. But we also have 
\[
\EE h_0(t,x+\theta t)=\EE h_0(t,0),
\]
as $\{h_0(t,x)-h_0(t,0)\}_{x\in\R}$ is a two-sided Brownian motion, thus, (i) is proved. 

(ii) By \eqref{e.331}, we have $\Var\, h_\theta(t,0)=\Var \,
h_0(t,\theta t)$.  Since $h_0(t,\theta
t)-h_0(t,0)\stackrel{\text{law}}{=} N(0,|\theta| t)$, we complete the
proof of (ii) by the triangle inequality.

(iii) We write the probability explicitly and change variables to obtain
\[
\begin{aligned}
\Pb_\theta^x(B_t\in A)=&\EE \frac{\int \cZ_t(x,y)e^{W_\theta(y)}1_{\{y\in A\}} dy}{\int \cZ_t(x,y)e^{W_\theta(y)}  dy}\\
=&\EE \frac{\int \cZ_t(x,x+\theta t +y)e^{W_\theta(x+\theta t+y)}1_{\{x+\theta t +y\in A\}} dy}{\int \cZ_t(x,x+\theta t+y)e^{W_\theta(x+\theta t +y)}  dy}\\
=&\EE \frac{\int \cZ_t(0,\theta t +y)e^{W(y)+\theta y}1_{\{x+\theta t +y\in A\}} dy}{\int \cZ_t(0,\theta t+y)e^{W(y)+\theta y}  dy},
\end{aligned}
\]
where in the last ``='' we used the stationarity. By the time reversal we have $\{\cZ_t(0,x)\}_{x\in\R}\stackrel{\text{law}}{=}\{\cZ_t(x,0)\}_{x\in\R}$, so the above probability can be written as
\[
\Pb_\theta^x(B_t\in A)=\EE \frac{\int \cZ_t(\theta t +y,0)e^{W(y)+\theta y}1_{\{x+\theta t +y\in A\}} dy}{\int \cZ_t(\theta t+y,0)e^{W(y)+\theta y}  dy}.
\]
Similar to \eqref{e.331}, we have 
\[
\{\cZ_t(\theta t+y,0)e^{\theta y+\frac12\theta^2 t}\}_{t>0,y\in\R}\stackrel{\text{law}}{=}\{\cZ_t(y,0)\}_{t>0,y\in\R},
\]
using which we rewrite the probability as 
\[
\Pb_\theta^x(B_t\in A)=\EE \frac{\int \cZ_t(y,0)e^{W(y)}1_{\{x+\theta t +y\in A\}} dy}{\int \cZ_t( y,0)e^{W(y)}  dy}.
\]
Using the time reversal again, we complete the proof of (iii).

(iv) This is similar to the discussion in \eqref{e.9191}. Using the Green's function of the SHE, we can write 
\begin{equation}\label{e.htheta}
h_\theta(t,x)= \log \int_{\R} \cZ_{t}(x,y)e^{W(y)+\theta y}dy.
\end{equation}
By  a straightforward calculation, using the representation
\eqref{e.htheta} and the definition of $p^0_\theta$ in
\eqref{e.defptheta},  we get
\begin{equation}\label{e.varBt}
\partial_\theta^2 h_\theta (t,0)=\int y^2 p^0_\theta(t,y)dy- ( \int yp^0_\theta(t,y)dy)^2,
\end{equation}
and the conclusion of part (iv) is a consequence of the Jensen inequality.
\end{proof}

\section{Upper bound}
\label{s.upbd}
%
%
%

The goal of this section is to show the upper bound
\begin{equation}\label{e.upbd}
\Var\, h_0(t,0) \leq Ct^{2/3}, \quad\quad t\geq1.
\end{equation}

We have the following lemma:
\begin{lemma}\label{l.totalvar}
For any $t>0$, we have \begin{equation}\label{e.351}
\EE \int y^2 \ssp(t,y)dy=t+\EE ( \int y\ssp(t,y)dy)^2
\end{equation}
\end{lemma}

\begin{proof}
The result is  a direct consequence of formula \eqref{e.varBt} used
for $\theta=0$ and part (i)   of
Proposition~\ref{l.meanhtheta}. 
\end{proof}

For the polymer endpoint $B_t$, there are two sources of randomnesses:
(i) the random environment $(\xi,W)$; (ii) for each realization of  the random environment, $B_t$ is sampled from the Gibbs measure. Thus,
the equation \eqref{e.351} can be viewed as a total variance formula:
the l.h.s.   is the total variance of $B_t$, $t$ is the expectation of
the quenched variance, and $\EE ( \int y\ssp(t,y)dy)^2$ is the
variance of the quenched expectation (the mean vanishes since
  $y\mapsto\EE\ssp(t,y)$ is even). As the total variance $\EE \int y^2 \ssp(t,y)dy$ is expected to be of order $t^{4/3}\gg t$, we see the main contribution must come from the variance of the quenched mean. This is consistent with the localization behavior of the polymer paths \cite{fisher}.

To estimate $\EE ( \int y\ssp (t,y)dy)^2$, we need

\begin{lemma}\label{l.varquenchedmean}
For any $\delta>0$ and $t>0$, we have 
\begin{equation}
\label{010403-22}
\left\|\int y \ssp(t,y)dy\right\|_2 \leq 4\delta^{-1}\sqrt{\Var\, h_0(t,0)}+2\sqrt{\delta^{-1} t}+\delta t.
\end{equation}
\end{lemma}

\begin{proof}
First, we recall that 
\[
\int y\ssp(t,y)dy =\partial_\theta h_\theta(t,0)\,\big|_{\theta=0}.
\]
By convexity of $h_{\theta}(t,0)$ in $\theta$  (part (iv) of Proposition~\ref{l.meanhtheta}) for any $\delta>0$  we have
\[
\left|\int y\ssp(t,y)dy\right|\leq \frac{1}{\delta} | h_{\delta}(t,0)- h_0(t,0)|+\frac{1}{\delta} | h_{-\delta}(t,0)- h_0(t,0)|.
\]
We remove the mean on the r.h.s. to further obtain 
\[
\begin{aligned}
\left|\int y \ssp(t,y)dy\right| \leq &  \frac{1}{\delta} | \hat{h}_{\delta}(t,0)- \hat{h}_0(t,0)|+\frac{1}{\delta} | \hat{h}_{-\delta}(t,0)- \hat{h}_0(t,0)|\\
&+\frac{1}{\delta}|\EE h_\delta(t,0)-\EE h_0(t,0)|+\frac{1}{\delta}|\EE h_{-\delta}(t,0)-\EE h_0(t,0)|,
\end{aligned}
\]
where we have denoted $\hat{h}=h-\EE h$. For the second line on the
r.h.s., which is purely deterministic, by part (i) of Proposition~\ref{l.meanhtheta}, we
have $\EE h_{\pm\delta}(t,0)-\EE h_0(t,0)=\frac12\delta^2 t$, which
leads to the upper bound of $\delta^{-1} \delta^2 t=\delta
t$. Applying the triangle inequality we have 
\begin{equation}
\label{020403-22}
\left\|\int y \ssp(t,y)dy\right\|_2 \leq \frac{1}{\delta} \bigg(\|\hat{h}_{\delta}(t,0)\|_2+2\|\hat{h}_{0}(t,0)\|_2+\|\hat{h}_{-\delta}(t,0)\|_2\bigg)+\delta t.
\end{equation}
From   part (ii) of Proposition~\ref{l.meanhtheta}, we have 
\begin{equation}
\label{030403-22}
\|\hat{h}_{\pm\delta}(t,0)\|_2\leq \|\hat{h}_{0}(t,0)\|_2+\sqrt{ \delta t}.
\end{equation}
Putting together  \eqref{020403-22} and \eqref{030403-22} we
  conclude \eqref{010403-22}.
\end{proof}

Now we can complete the proof of the upper bound in Theorem~\ref{t.bqs}:

\begin{proof}[Proof of \eqref{e.upbd}]
Applying Proposition~\ref{p.inte}, Jensen's inequality, and Lemma~\ref{l.totalvar}, we have 
\[
\Var \,h_0(t,0) \leq \sqrt{\EE \int y^2 \ssp (t,y)dy} =\sqrt{t+\EE (\int y\ssp (t,y)dy)^2}.
\]
Further applying Lemma~\ref{l.varquenchedmean}, we derive that
\[
\Var \,h_0(t,0) \leq \sqrt{t}+4\delta^{-1}\sqrt{\Var \, h_0(t,0)}+2\sqrt{\delta^{-1} t}+\delta t
\]
for all $\delta>0$. Choosing $\delta=t^{-1/3}$, we have for $t\geq1$ that 
\[
\Var \, h_0(t,0) \leq C(t^{1/3}\sqrt{\Var \, h_0(t,0)}+t^{2/3}),
\]
where $C>0$ is some universal constant. This is equivalent with 
\[
(\sqrt{\Var \, h_0(t,0)}-\frac12Ct^{1/3})^2 \leq (\frac14C^2+C)t^{2/3}.
\]
Hence the proof is complete.
\end{proof}

\section{Lower bound}
\label{s.lowbd}

Recall that 
$\Var \, h_0(t,0)=\E_0^0|B_t|$, see \eqref {050403-22}. In the present
section 
we shall show that there exists $C>1$, for which
 \begin{equation}\label{e.lowbd}
 \Var \,h_0(t,0)\geq C^{-1}t^{2/3},\quad t\ge1.
 \end{equation}

Let $u>0$ and $\theta>0$ be two constants to be determined later on.
Fix $t>1$ and define 
\begin{equation}\label{e.defn}
n=u+\theta t.
\end{equation}
The idea is to estimate the two probabilities $\Pb_\theta^{0}(B_t>n)$ and $\Pb_\theta^{0}(B_t\leq n)$ separately from above by $t^{-2/3}\Var \, h_0(t,0)$. The first probability can be easily estimated by the Chebyshev inequality:
\begin{lemma}\label{l.p1}
We have 
\[
\Pb_\theta^{0}(B_t>n)\leq \frac{\Var\, h_0(t,0)}{u}.
\]
\end{lemma}

\begin{proof} 
By Proposition~\ref{l.meanhtheta}, we have 
\[
\Pb_\theta^{0}(B_t>n)=\Pb_0^{0}(B_t+\theta t>n)=\Pb_0^0(B_t>u ).
\]
By the Markov inequality we have 
\[
\Pb_0^0(B_t>u) \leq \frac{\E_0^0|B_t|}{u }=\frac{\Var\, h_0(t,0)}{u},
\]
which completes the proof.
\end{proof}

To estimate the other probability, inspired by the argument in \cite{timo3}, we introduce another initial data $\tilde{W}_\theta$ which is a perturbation of $W$ in $[0,n]$:
\begin{equation}\label{e.deftW}
\tilde{W}_\theta(x)=(W(x)+\theta x)\1_{\{x\in[0,n]\}}+(W(x)+\theta n)\1_{\{x>n\}}+W(x)\1_{\{x<0\}}.
\end{equation}
In other words, we add a drift $\theta$ in the interval $[0,n]$.  Define $\tilde{h}_\theta$ as the solution to the KPZ equation started from $\tilde{W}_\theta$ driven by the same noise $\xi$, i.e., 
\[
\tilde{h}_\theta(t,x)= \log \int \cZ_t(x,y)e^{\tilde{W}_\theta(y)}dy. 
\]

Let $\cX$ be a random variable with exponential distribution of
parameter $1$ that is independent of  the random element $(\xi,W)$,
where $\xi$ is the spacetime white noise and $W$ is the two-sided Brownian motion. 

%
%
%

{The idea is to compare $h_\theta(t,0)$ with $\tilde{h}_\theta(t,0)$. By construction, we have  $W_\theta(x)\leq \tilde{W}_\theta (x)$ in the region of $x\leq n$, therefore, in the event of $B_t\leq n$, we do not expect that $h_\theta(t,0)$ to be much larger than $\tilde{h}_\theta(t,0)$. 
The following key lemma   makes the heuristics precise. It corresponds to \cite[Lemma 4.1]{timo3} in the context of ASEP, which was proved through a coupling argument.}

\begin{lemma}\label{l.keyl}
We have
\[
\Pb_\theta^{0}(B_t\leq n) \leq \PP(h_\theta(t,0)-\tilde{h}_\theta(t,0)\leq \cX).
\]
\end{lemma}

\begin{proof}
First, we can write
\[
\Pb_\theta^{0}(B_t\leq n)=\EE \frac{\int \cZ_t(0,y)e^{W_\theta(y)}\1_{\{y\leq n\}}dy}{\int \cZ_t(0,y)e^{W_\theta(y)}dy}=\EE \frac{1}{1+\X},
\]
with 
\[
\X:=\frac{\int \cZ_t(0,y)e^{W_\theta(y)}\1_{\{y> n\}}dy}{\int \cZ_t(0,y)e^{W_\theta(y)}\1_{\{y\leq n\}}dy}>0.
\]
On the other hand, we have
\[
\begin{aligned}
\Y&:=h_\theta(t,0) -\tilde{h}_\theta(t,0)=\log \frac{\int \cZ_t(0,y)e^{W_\theta(y)}dy}{\int \cZ_t(0,y)e^{\tilde{W}_\theta(y)}dy}\\
&\leq  \log \frac{\int  \cZ_t(0,y)e^{W_\theta(y)}\1_{\{y>n\}}dy+\int  \cZ_t(0,y)e^{W_\theta(y)}\1_{\{y\leq n\}}dy}{\int  \cZ_t(0,y)e^{\tilde{W}_\theta(y)}\1_{\{y\leq n\}}dy}.
\end{aligned}
\]
By construction, we have $W_\theta (y)\leq \tilde{W}_\theta(y)$ when
$y\leq n$, which implies
that
\begin{align}\label{010103-22}
 \Y\le 
       \log \frac{\int  \cZ_t(0,y)e^{W_\theta(y)}\1_{\{y>n\}}dy+\int
   \cZ_t(0,y)e^{W_\theta(y)}\1_{\{y\leq n\}}dy}{\int
      \cZ_t(0,y)e^{{W}_\theta(y)}\1_{\{y\leq n\}}dy}
     = \log(1+\X). 
\end{align}
This, in turn, implies that 
\[
\begin{aligned}
\EE \frac{1}{1+\X}=&\int_0^1\PP[\X<z^{-1}-1]
dz=\int_0^1\PP[\log(1+\X)<\log z^{-1}] dz\\
\leq & \int_0^1 \PP[\Y<\log z^{-1}]dz.
\end{aligned}
\]
An elementary calculation gives 
\[
\PP[\Y<\cX]=\int_0^\infty \PP[\Y<x]e^{-x}dx=\int_0^1 \PP[\Y<\log z^{-1}] dz,
\]
which completes the proof.
\end{proof}

It remains to estimate $\PP(h_\theta(t,0)-\tilde{h}_\theta(t,0)\leq \cX)$. For any $c_1,c_2,c_3\in\R $ satisfying $c_1=c_2+c_3$, we have 
\begin{equation}\label{e.311}
\begin{aligned}
&\PP(h_\theta(t,0)-\tilde{h}_\theta(t,0)\leq \cX)\\
&\leq  \PP(\{h_\theta(t,0)\leq  c_1\}\cup\{\tilde{h}_\theta(t,0)> c_2\}\cup \{\cX> c_3\})\\
& \leq  \PP(h_\theta(t,0)\leq  c_1)+\PP(\tilde{h}_\theta(t,0)> c_2)+\PP(\cX> c_3).
\end{aligned}
\end{equation}

Through the following lemmas, we estimate each probability from the above display separately.
Since $\cX$ is of exponential distribution with parameter $1$, we have
\begin{lemma}
For any $c_3>0$, we have $\PP(\cX> c_3)= e^{-c_3}$.
\end{lemma}

\begin{lemma}
For $c_1<\EE h_0(t,0)+\frac12\theta^2 t$, we have 
\begin{equation}
\label{060402-22}
\PP(h_\theta(t,0)\leq c_1)\leq \frac{\sqrt{\Var\, h_0(t,0)}+\sqrt{\theta t}}{\EE h_0(t,0)+\frac12\theta^2 t-c_1}.
\end{equation}
\end{lemma}

\begin{proof}
First, we write
\[
\PP(h_\theta(t,0)\leq c_1)=\PP(\hat{h}_\theta(t,0)\leq c_1-\EE h_\theta(t,0)).
\]
where, as we recall $\hat{h}_\theta(t,0) :=h_\theta(t,0)-\EE h_\theta(t,0)$.
By part (i) of Proposition~\ref{l.meanhtheta}, we have $\EE h_\theta(t,0)=\EE
h_0(t,0)+\frac12\theta^2 t$. Under the assumption on $c_1$, we can 
apply the Markov and Jensen inequalities  to conclude that 
\begin{align*}
  &\PP(-\hat{h}_\theta(t,0)\geq \EE  h_\theta(t,0)-c_1)\le \frac{\EE |
    \hat h_\theta(t,0)|}{\EE  h_0(t,0) +\frac12\theta^2 t-c_1} 
  \\
  &
    \le \frac{\sqrt{\Var\, h_\theta(t,0)}}{\EE  h_0(t,0) +\frac12\theta^2 t-c_1}.
\end{align*} 
Furthermore, by part (ii) of Proposition ~\ref{l.meanhtheta},
\[
\frac{\sqrt{\Var\, h_\theta(t,0)}}{\EE h_0(t,0) +\frac12\theta^2 t-c_1} \leq \frac{\sqrt{\Var \, h_0(t,0)}+\sqrt{\theta t}}{\EE h_0(t,0) +\frac12\theta^2 t-c_1},
\]
and \eqref{060402-22} follows.
\end{proof}


\begin{lemma}
Assuming $c_2>\EE h_0(t,0)$, we have 
\[
\PP(\tilde{h}_\theta(t,0)> c_2) \leq e^{\frac12\theta^2 n}\frac{\sqrt{\Var \, h_0(t,0)}}{c_2-\EE h_0(t,0)}. 
\]
\end{lemma}

\begin{proof}
Recall that $\tilde{h}_\theta$ starts from $\tilde{W}_\theta$ which
only has a positive drift $\theta$ in $[0,n]$, applying the Girsanov theorem, we can write 
\[
\begin{aligned}
\PP(\tilde{h}_\theta(t,0)>c_2)=\EE \1_{\{\tilde{h}_\theta(t,0)> c_2\}}=\EE \1_{\{h_0(t,0)> c_2\}} \mathscr{G},
\end{aligned}
\]
with the Radon-Nikodym derivative $\mathscr{G}=e^{\theta
  W(n)-\frac12\theta^2 n}$. Applying the Cauchy-Schwarz inequality, we have 
\begin{equation}
\label{070403-22}
\PP(\tilde{h}_\theta(t,0)> c_2) \leq \sqrt{\PP(h_0(t,0)> c_2)} \sqrt{\EE \mathscr{G}^2}.
\end{equation}
A direct calculation gives $\EE \mathscr{G}^2= \EE e^{2\theta
  W(n)-\theta^2 n}=e^{\theta^2n}$. For the probability appearing on
the right hand side of \eqref{070403-22},   an
application of the Chebyshev inequality gives 
\[
\PP(h_0(t,0)>c_2) \leq \frac{\Var\, h_0(t,0)}{(c_2-\EE h_0(t,0))^2}, 
\]
which completes the proof.
\end{proof}

To simplify the notation, from now on we denote 
\[
c(t)=\EE h_0(t,0), \quad\quad \psi(t)=\Var\, h_0(t,0).
\]
Combining the above three lemmas, we have
\begin{equation}\label{e.312}
\begin{aligned}
\Pb_\theta^{0}(B_t\leq n)&\leq  \PP(h_\theta(t,0)\leq c_1)+\PP(\tilde{h}_\theta(t,0)> c_2)+\PP(\cX> c_3)\\
&\leq \frac{\sqrt{\psi(t)}+\sqrt{\theta t}}{c(t)+\frac12\theta^2 t-c_1}+e^{\frac12\theta^2 n}\frac{\sqrt{\psi(t)}}{c_2-c(t)} +e^{-c_3},
\end{aligned}
\end{equation}
provided that $c_1=c_2+c_3$ and 
\[
c_1<c(t) +\frac12\theta^2 t, \quad\quad c_2>c(t), \quad\quad c_3>0.
\]

Now we can finish the proof of the lower bound. 

\begin{proof}[Proof of \eqref{e.lowbd}]
Suppose that $M>\lambda>0$ and $\lambda^2>4$. They are  to be further adjusted later on. Let
\[
c_1=c(t)+2t^{1/3}, \quad c_2=c(t)+t^{1/3}, \quad c_3=t^{1/3},  
\]
and
\[
\theta =\lambda t^{-1/3}, \quad n=Mt^{2/3}, \quad u=(M-\lambda)t^{2/3}.
\]
They obviously   satisfy  \eqref{e.defn}.
From Lemma~\ref{l.p1} and \eqref{e.312}, we have
\[
\begin{aligned}
&\Pb_\theta^0(B_t>n)\leq \frac{\psi(t)}{(M-\lambda) t^{2/3}},\\
&\Pb_\theta^0(B_t\leq n)\leq  \frac{1}{\frac12\lambda^2-2} \sqrt{\frac{\psi(t)}{t^{2/3}}}+e^{\frac12\lambda^2M} \sqrt{\frac{\psi(t)}{t^{2/3}}}+\frac{\sqrt{\lambda}}{\frac12\lambda^2 -2}+e^{-t^{1/3}}.
\end{aligned}
\]
Adding the above two inequalities, we obtain
\begin{align}
\label{030303-22}
1\leq 
a \frac{\psi(t)}{  t^{2/3}}+b \sqrt{\frac{\psi(t)}{t^{2/3}}}+\frac{\sqrt{\lambda}}{\frac12\lambda^2 -2}+e^{-t^{1/3}},
\end{align}
where
\begin{align*}
a:=\frac{1}{M-\lambda},\qquad
b:=
  \frac{1}{\frac12\lambda^2-2} 
+e^{\frac12\lambda^2M}.
\end{align*}
Fixing the parameters $\lambda,M$ so that 
$$
M>\lambda, \quad \lambda^2>4\quad\mbox{and}\quad 
\frac{\sqrt{\lambda}}{\frac12\lambda^2 -2}<1,
$$
 we conclude from \eqref{030303-22} that $\liminf_{t\to\infty} \psi(t)t^{-2/3}>0$. 
\end{proof}

\appendix

\section{Proof of \eqref{e.decayco}}

For the convenience of readers, we provide a self-contained proof of the covariance decay result in \eqref{e.decayco}.  Recall that 
\begin{equation}
\label{030903-22}
\cH(t,x)=h_0(t,x)-W(x)= \log \int \cZ_t(x,y)e^{W(y)-W(x)}dy.
\end{equation}
Fix $t>0$, the goal in this section is to show that 
\begin{equation}\label{e.decayco1}
\cov[\cH(t,0),\cH(t,x)]\to0, \quad\quad \mbox{ as } |x|\to\infty.
\end{equation}

There are two independent Gaussian
processes appearing in \eqref{030903-22}: the noise $\xi$ and the two-sided Brownian motion
$W$. Denote by $\E_{\xi}$ and $\E_W$ the expectations on 
$\xi$ and $W$ respectively. Recall that we used $\D$ to denote the Malliavin derivative with
respect to $W'$, and $\EE$ is the total expectation: $\EE=\E_\xi\E_W$. From now on we will use $\mathscr{D}$ to denote the
Malliavin derivative with respect to $\xi$.  We can write
\begin{align}
\label{010903-22}
&\cov[\cH(t,0),\cH(t,x)]  =\EE \Big\{\Big[\cH(t,x)-\E_{\xi} \cH(t,x) \Big]\Big[\cH(t,0)-\E_{\xi} \cH(t,0)\Big]\Big\}\\
&
+
 \E_W\Big\{\Big[ \E_{\xi}
  \cH(t,x)-\EE  \cH(t,x)\Big]\Big[ \E_{\xi} \cH(t,0)-\EE
  \cH(t,0)\Big]\Big\}.\notag
\end{align}
Fix the realization of $W$ and use the
Clark-Ocone formula for the $\xi$ noise   (see  \cite[Proposition
6.3]{CKNP19}),  we can write 
\begin{equation}\label{e.co11}
	\cH(t,x)-\E_{\xi}\cH(t,x)=\int_0^t \int_{\R} \E_{\xi}[ \mathscr{D}_{s,z}\cH(t,x)\mid\F_s]\xi(s,z)dzds,
\end{equation}
where 
$\{\F_s\}_{s\geq0}$ is the natural filtration corresponding to $\xi$.
On the other hand, we can use Clark-Ocone again to express  $\E_{\xi}
  \cH(t,x)-\EE  \cH(t,x)$  and we get
\begin{align}
\label{020903-22}
&
\Big|\E_W\Big\{\Big[ \E_{\xi}
  \cH(t,x)-\EE  \cH(t,x)\Big]\Big[ \E_{\xi} \cH(t,0)-\EE
  \cH(t,0)\Big]\Big\}\Big| \notag\\
&
\le \int \|\D_z \cH(t,0)\|_2 \|\D_z\cH(t,x)\|_2 dz.                                                                     
\end{align}
Using \eqref{e.co11} and \eqref{020903-22}, we can estimate the
expression \eqref{010903-22}
with the help of the Cauchy-Schwarz inequality and get
\[
\begin{aligned}
|\cov[\cH(t,0),\cH(t,x)]| \leq &\int \|\D_z \cH(t,0)\|_2 \|\D_z\cH(t,x)\|_2 dz\\
+&\int_0^t\int \|\cD_{s,z}\cH(t,0)\|_2 \|\cD_{s,z}\cH(t,x)\|_2 dzds=: I_1+I_2.
\end{aligned}
\]

Before estimating $I_1,I_2$, we introduce another notation, the propagator of SHE from $(s,z)$ to $(t,x)$, which is the solution to 
\[
\partial_t \cZ_{t,s}(x,z)=\frac12\Delta_x \cZ_{t,s}(x,z)+\xi(t,x)\cZ_{t,s}(x,z), \quad\quad t>s
\]
and $\cZ_{s,s}(x,z)=\delta(x-z)$. For the propagator, we have the moment estimates \cite[Theorem 2.4, Example 2.10]{chenle}: for any $p\geq1$ and $0\leq s<t<T$, there exists a constant $C=C(p,T)>0$ such that 
\begin{equation}\label{e.mmshe}
\|\cZ_{t,s}(x,y)\|_p \leq C(t-s)^{-1/2}e^{-\frac{(x-y)^2}{C (t-s)}}.
\end{equation}

Throughout the rest of the proof, $C>0$ is some constant that depends only on $t>0$.

(i) Estimates on $I_1$. For any $x,z\in\R$, we have 
\[
\D_z\cH(t,x) =\frac{\int \cZ_t(x,y)e^{W(y)-W(x)} [\1_{\{x<z<y\}}-\1_{\{y<z\leq x\}}]dy}{\int \cZ_t(x,y)e^{W(y)-W(x)}dy}.
\]
By the moment estimate in \eqref{e.mmshe} and a proof that is very similar to the one for Lemma~\ref{l.bdheatk}, we have 
\[
\begin{aligned}
\|\D_z\cH(t,x)\|_2 &\leq  C\int e^{-\frac{(x-y)^2}{Ct}} e^{C|y-x|}[\1_{\{x<z<y\}}+\1_{\{y<z\leq x\}}]dy\\
&\leq  C\int e^{-\frac{(x-y)^2}{Ct}} e^{C|y-x|}\1_{\{|y-x|\geq |z-x|\}} dy\\
&=C\int e^{-\frac{y^2}{Ct}} e^{C|y|}\1_{\{|y|\geq |z-x|\}} dy=:\phi(|x-z|).
\end{aligned}
\]
Then it is straightforward to check that 
\[
I_1\leq \int\phi(|z|)\phi(|x-z|)dz\to0, \quad\quad \mbox{ as } |x|\to\infty.
\]

(ii) Estimates on $I_2$. For the Malliavin derivative with respect to $\xi$, we apply \cite[Theorem 3.2]{CKNP20} to obtain
\[
\cD_{s,z}\cH(t,x)= \frac{ \cZ_{t,s}(x,z)\int\cZ_s(z,y)e^{W(y)-W(x)}dy}{\int \cZ_t(x,y)e^{W(y)-W(x)}dy},
\]
Applying again a  proof that is similar to the one for Lemma~\ref{l.bdheatk}, we have 
\[
\begin{aligned}
\|\cD_{s,z}\cH(t,x)\|_2&\leq  C(t-s)^{-1/2}e^{-\frac{(x-z)^2}{C(t-s)}}\int s^{-1/2}e^{-\frac{(z-y)^2}{Cs}}e^{C|y-x|}dy\\
&\leq  C(t-s)^{-1/2}e^{-\frac{(x-z)^2}{C(t-s)}}e^{C|x-z|}=:\varphi_{t-s}(|x-z|).
\end{aligned}
\]
This implies that 
\[
I_2\leq  \int_0^t\int \varphi_{t-s}(|x-z|)\varphi_{t-s}(|z|)dzds.
\]
From the above expression, it is another straightforward calculation to conclude that $I_2\to0$ as $|x|\to\infty$. This finishes the proof of \eqref{e.decayco1}.


\begin{thebibliography}{99}


\bibitem{acq}
G. Amir, I. Corwin, and J. Quastel, ``Probability distribution of the free energy of the
continuum directed random polymer in 1+1 dimensions'', Comm. Pure Appl. Math., 64 (2011),
466–537.

\bibitem{bakhtin}
Y.~Bakhtin,  and L.~Li, ``Thermodynamic limit for directed polymers and stationary solutions of the Burgers equation'',  Communications on Pure and Applied Mathematics 72.3 (2019): 536-619.




\bibitem{timo4}
M.~Balázs, E.~Cator, and Timo Sepp\"al\"ainen, ``Cube root fluctuations for the corner growth model associated to the exclusion process'', Electronic Journal of Probability 11 (2006): 1094-1132.


\bibitem{bqs}
M. Bal\'azs, J. Quastel, and T. Sepp\"al\"ainen, ``Fluctuation exponent of the KPZ/stochastic
Burgers equation'', J. Amer. Math. Soc., 24 (2011), 683–708.




\bibitem{timo3}
M.~Bal\'azs, and T.~Sepp\"al\"ainen, ``Order of current variance and diffusivity in the asymmetric simple exclusion process'', Annals of mathematics (2010): 1237-1265.


\bibitem{bc1}
G.~Barraquand,  and I.~Corwin, ``Random-walk in beta-distributed random environment'', Probability Theory and Related Fields 167.3 (2017): 1057-1116.

\bibitem{BG97}
L.~Bertini and G.~Giacomin, ``Stochastic Burgers and KPZ equations
  from particle systems'', Comm. Math. Phys., 183 (1997), 571--607.



%
%
%

\bibitem{bc}
A.~Borodin, and I.~Corwin, ``Macdonald processes'', Probability Theory and Related Fields 158.1 (2014): 225-400.

\bibitem{bcf}
A. Borodin, I. Corwin, and P. Ferrari, ``Free energy fluctuations for directed polymers in
random media in 1+1 dimension'', Comm. Pure Appl. Math., 67 (2014), 1129–1214.


\bibitem{BCFV15}
 A.~Borodin, I.~Corwin, P.~Ferrari, and B.~Vet{\H o}, ``Height
  fluctuations for the stationary {KPZ} equation'', Math. Phys. Anal. Geom., 18
  (2015), Art. 20, 95.


\bibitem{bcr}
A.~Borodin, Alexei, I.~Corwin, and D.~Remenik, ``Log-gamma polymer free energy fluctuations via a Fredholm determinant identity'', Communications in Mathematical Physics 324.1 (2013): 215-232.


%
%
%
%


\bibitem{chenle}
L.~Chen and  R.~Dalang, ``Moments and growth indices for the nonlinear stochastic heat equation with rough initial conditions'', Ann. Probab. 43 (6) 3006 - 3051, 2015.

\bibitem{CKNP19}
L.~Chen, D.~Khoshnevisan, D.~Nualart, and F.~Pu, { Spatial
  ergodicity for SPDEs via Poincar\'e-type inequalities}
 Electron. J. Probab. 26 (2021), Paper No. 140, 37 pp.


\bibitem{CKNP20}
L.~Chen, D.~Khoshnevisan, D.~Nualart, and F.~Pu, ``Central limit
  theorems for spatial averages of the stochastic heat equation via
  Malliavin-Stein's method'', Aug. 2020, arXiv preprint
  \href{https://arxiv.org/abs/2008.02408v1}{2008.02408v1}.




%
%

\bibitem{corwinh}
I.~Corwin,and A.~Hammond, ``KPZ line ensemble'', Probability Theory and Related Fields 166.1 (2016): 67-185.

\bibitem{haoshen}
I.~Corwin, T.~Sepp\"al\"ainen, and H.~Shen, ``The strict-weak lattice polymer'', Journal of Statistical Physics 160.4 (2015): 1027-1053.





%
%
%
%




\bibitem{corwin2012kardar}
I.~Corwin, ``The Kardar--Parisi--Zhang equation and universality
  class'', Random matrices: Theory and applications, 1 (2012), p.~1130001.
  
  
\bibitem{dunlap}
A.~Dunlap, C.~Graham, and L.~Ryzhik, ``Stationary solutions to the stochastic Burgers equation on the line'', Communications in Mathematical Physics 382.2 (2021): 875-949.
  
  \bibitem{fisher}
  D.~Fisher, and D.~Huse, ``Directed paths in a random potential'', Physical Review B 43.13 (1991): 10728.
  
  
\bibitem{timo2}
Gregorio Moreno Flores, T.~Seppäläinen, and B.~Valk\'o, ``Fluctuation exponents for directed polymers in the intermediate disorder regime'', Electronic Journal of Probability 19 (2014): 1-28.


%
%




\bibitem{funaki}
T.~Funaki and J.~Quastel, ``KPZ equation, its renormalization and invariant measures'', Stochastic Partial Differential Equations: Analysis and Computations 3.2 (2015), pp.~159--220.


%
%
  
\bibitem{khoa}
Y.~Hu and K.~L\^e, ``Asymptotics of the density of parabolic Anderson random fields'',  Ann. Inst. H. Poincar\'e Probab. Statist. 58 (1) 105 - 133, February 2022.



%

\bibitem{soso2}
B.~Landon, C.~Noack and Philippe Sosoe, ``KPZ-type fluctuation bounds for interacting diffusions in equilibrium'', arXiv preprint arXiv:2011.12812 (2020).


\bibitem{pimentel1}
S.~L\'opez, and Leandro PR Pimentel, ``On the two-point function of the one-dimensional KPZ equation'', arXiv preprint arXiv:2208.14987 (2022).


\bibitem{MT}
C.~Maes, T.~Thiery,``Midpoint Distribution of Directed Polymers in the Stationary Regime: Exact Result Through Linear Response'', J Stat Phys 168, 937–963 (2017).


\bibitem{MQR20}
K.~Matetski, J.~Quastel, and D.~Remenik, ``The KPZ fixed point'',
  Oct. 2020, arXiv preprint
  \href{https://arxiv.org/abs/1701.00018v3}{1701.00018v3}.



\bibitem{soso1}
C.~Noack, and P.~Sosoe, ``Central moments of the free energy of the O'Connell-Yor polymer'', arXiv preprint arXiv:2003.01170 (2020).


\bibitem{nualart} D.~Nualart, ``The Malliavin Calculus and Related Topics'', Springer New York 2000.


\bibitem{pimentel}
 Leandro PR Pimentel, ``Integration by Parts and the KPZ Two-Point Function'', The Annals of Probability 50.5 (2022): 1755-1780.


\bibitem{Qua12}
J.~Quastel, ``Introduction to KPZ'', in Current Developments in
  Mathematics, 2011, Int. Press, Somerville, MA, 2012, 125--194.

\bibitem{QS20}
J.~Quastel and S.~Sarkar, ``Convergence of exclusion processes and
  KPZ equation to the KPZ fixed point'', Oct. 2020, arXiv preprint
  \href{https://arxiv.org/abs/2008.06584v4}{2008.06584v4}.

\bibitem{qs}
J. Quastel and H. Spohn, ``The one-dimensional KPZ equation and its universality class'', J.
Stat. Phys., 160 (2015), 965–984.

\bibitem{SS10a}
T.~Sasamoto and H.~Spohn, ``Exact height distributions for the {KPZ}
  equation with narrow wedge initial condition'', Nuclear Phys. B, 834 (2010),
  523--542.
  
  
  

\bibitem{timo}
T. Sepp\"al\"ainen, ``Scaling for a one-dimensional directed polymer with boundary conditions'', Ann. Probab. 40, 19–73 (2012).


\bibitem{timo1}
T. Sepp\"al\"ainen, B. Valk\'o, ``Bounds for scaling exponents for a 1+1 dimensional directed polymer in a Brownian environment'', ALEA 7, 451-476 (2010)

\bibitem{Vir20}
B.~Vir\'ag, ``The heat and the landscape I'', Aug. 2020, arXiv preprint
  \href{https://arxiv.org/abs/2008.07241v1}{2008.07241v1}.


%
%
%


\end{thebibliography}
\end{document}